\documentclass[12pt,a4paper]{amsart}
\usepackage{dsfont}
\usepackage{graphicx}
\usepackage{amssymb}
\usepackage[all]{xy}
\usepackage{amsmath}
\usepackage{tikz}
\usepackage[french,english]{babel}
\usepackage{amssymb,amsmath,amsfonts,amsthm,amscd}
\usepackage{indentfirst,graphicx}
\usepackage[font={scriptsize},captionskip=5pt]{subfig}
\usepackage[T1]{fontenc}

\newtheorem{theorem}{Theorem}[section]
\newtheorem{corollary}[theorem]{Corollary}
\newtheorem{proposition}[theorem]{Proposition}
\newtheorem{lemma}[theorem]{Lemma}

\newtheorem{definition}[theorem]{Definition}

\newtheorem{remark}[theorem]{Remark}

\newcommand{\cA}{{\mathcal A}}

\newcommand{\cD}{{\mathcal D}}
\newcommand{\cE}{{\mathcal E}}

\newcommand{\cG}{{\mathcal G}}

\newcommand{\cO}{{\mathcal O}}

\newcommand{\bC}{{\mathbb{C}}}

\newcommand{\bR}{{\mathbb{R}}}

\usepackage{fancyhdr}
\begin{document}

\title{REAL AND COMPLEX INTEGRAL CLOSURE, LIPSCHITZ EQUISINGULARITY AND APPLICATIONS ON  SQUARE MATRICES}

\author{Thiago F. da Silva}
\author{Nivaldo G. Grulha Jr.}
\author{Miriam S. Pereira}

\maketitle

\begin{center}
{	{\small \textit{Dedicated to Terence Gaffney and Maria Ruas, on the occasion of their 70th birthday, and to Marcelo Saia, on the occasion of his 60th birthday.}}}
\end{center}

\begin{abstract}
{\small Recently the authors investigated the Lipschitz triviality of simple germs of matrices. In this work, we improve some previous results and we present an extension of an integral closure result for the real setting. These tools are applied to investigate classes of square matrices singularities classified by Bruce and Tari.}
%an this work we investigate the Bi-Lipschitz equisingularity of 1-unfolding of EIDS from the canonical vector field, inspired by Gaffney's work.}
\end{abstract}

\let\thefootnote\relax\footnote{2010 \textit{{Mathematics Subjects Classification} 32S15, 14J17, 32S60.}
	
}

\thispagestyle{empty}

\section*{Introduction}

The study of Lipschitz equisingularity has risen from works of Zariski \cite{Za}, Pham \cite{Pham} and Teissier \cite{PT} and further developed by Parusi\'nski (\cite{PA1,PA2}), Gaffney (\cite{SG,G1,G2}), Fernandes, Ruas (\cite{FR2}) and others.

 In \cite{M1} Mostowski introduced a new technique for the study of this subject from the existence of Lipschitz vector fields. In general, this vector field is not canonical from the variety. Nevertheless, Gaffney \cite{G1} presented conditions to find a canonical Lipschitz vector field in the context of  a family of irreducible curves  {using the \textit{double} structure, defined for ideals in \cite{G2} and generalized for modules in \cite{SG}.}
 
Families of square matrices were first studied by Arnold in \cite{Arnold}, where the parametrised invertible matrices act by conjugation. Recently, many authors presented a series of interesting results about determinacy and classification using parametrised families or smooth changes of coordinates in the source of the germ  (\cite{Bruce}, \cite{BruceTari}, \cite{FK}, \cite{FN} and \cite{Miriam}).

More recently, Gaffney's result was extended in \cite{SGP}, where the authors presented conditions which ensure the canonical vector field is Lipschitz in the context of $1$-unfoldings of singularities of matrices, following the approach of Pereira and Ruas \cite{RP}.

In this work we prove a real version of the result proved in \cite{SGP} in order to investigate the Lipschitz triviality in the real case. Finally, we study some deformations of simple singularities classified by Bruce and Tari \cite{Bruce, BruceTari} in real and complex cases, using a similar approach as in \cite{SGP}. 

\section*{Acknowledgements}

%The authors are grateful to Terence Gaffney and Maria Aparecida Soares Ruas for the inspiration and support for this work, and to Anne Fr\"uhbis-Kr\"uger, for her comments and suggestions which provided the improvement of this paper, mainly on Theorem 2.7 and by the remark that apears here as Remark 2.8.

{The first author was supported by Proex ICMC/USP in a visit to S\~ao Carlos, where part of this work was developed. The second author was partially supported by FAPESP and CNPq.
 
\section{Notation and Background}

We start with some notation. Let $\mathbb{K}$ be a field which is $\mathbb{R}$ or $\mathbb{C}$ and let $\mathcal{R}$ be the group of diffeomorphisms $\mathbb{K}^r,0\to\mathbb{K}^r,0$. Let $\mathcal{H}$ denote the set of germs of smooth mappings $\mathbb{K}^r,0\to GL_n(V)\times GL_p(W)$, and $M$
the set of germs $F:\mathbb{K}^r,0\to \mbox{Hom}(V;W)$. The set $\mathcal{H}$ can be endowed with a group structure inherited from the product group in the target.

%Given two matrices, we want to compare their germs according to the following equivalence relation.

We define a notion of bi-Lipschitz equivalence between two matrices as in \cite{MP}.

%\begin{definition}\label{action}
%Let ${\mathcal{G}}={\mathcal{R}}\ltimes {\mathcal{H}}$ be the semi-direct product of ${\mathcal{R}}$ and ${\mathcal{H}}$. We say that two germs $F_1,
%\,\ F_2:\mathbb{K}^r,0\to \mbox{Hom}(V;W)$ are $\mathcal{G}$-Lipschitz equivalent if there exist a germ of a bi-Lipschitz homeomorphism \\$\phi:(\mathbb{K}^r,0)\rightarrow (\mathbb{K}^r,0)$ and germs of continuous mappings \\$X:(\mathbb{K}^r,0)\rightarrow \mbox{GL}_n(V)$, $Y:(\mathbb{K}^r,0)\rightarrow \mbox{GL}_p(W)$ such that $$F_1=X^{-1}(F_2\circ \phi^{-1})Y.$$
%\end{definition}

\begin{definition}\label{action}
Let ${\mathcal{G}}={\mathcal{R}}\ltimes {\mathcal{H}}$ be the semi-direct product of ${\mathcal{R}}$ and ${\mathcal{H}}$. We say that two germs $F_1,
\,\ F_2:\mathbb{K}^r,0\to \mbox{Hom}(V;W)$ are $\mathcal{G}$-Lipschitz equivalent if there exist a germ $\phi:(\mathbb{K}^r,0)\rightarrow (\mathbb{K}^r,0)$ of a bi-Lipschitz homeomorphism and germs of continuous mappings \linebreak $X:(\mathbb{K}^r,0)\rightarrow \mbox{GL}_n(V)$, $Y:(\mathbb{K}^r,0)\rightarrow \mbox{GL}_p(W)$ such that $$F_1=X^{-1}(F_2\circ \phi^{-1})Y.$$
\end{definition}

An element of $M$ can also be considered as a map $\mathbb{K}^r,0\to\mathbb{K}^N$, where we identify $\mbox{Hom}(V;W)$ with the $n\times p$ matrices, and $N= np$.

It is not difficult to see that $\mathcal{G}$ is one of Damon's
geometric subgroups of $\mathcal{K}$. As a consequence of Damon's result 
we can use the techniques of singularity theory. For instance, those
concerning finite determinacy (see \cite{Damon}, \cite{Miriam} and \cite{BruceTari}). 

It is possible to determine the tangent space to the orbit for the action of the group $\mathcal{G}$ on $M$. Given a matrix $F$, we write $F_{x(i)}$ for the matrix $\dfrac{\partial F}{\partial x_i}$ and we denote $\cE_r$ for the ring of smooth functions $\mathbb{K}^r, 0\to\mathbb{K}$. So the tangent space could be viewed as an $\cE_r$-submodule of $\cE_N$ spanned by the set of matrices $R_{il}$ (respectively $C_{jm}$) with $l^{\mbox{\tiny{th}}}$ row (respectively $m^{\mbox{\tiny{th}}}$ th column) the $i^{\mbox{\tiny{th}}}$  row of $F$ (respectively  $j^{\mbox{\tiny{th}}}$  column) and with zeros elsewhere, for $1\leq i,\,l\leq n$ and $1\leq j,\, m\leq p$ (see \cite{Damon}, \cite{Miriam} and \cite{BruceTari}).

\section{Real integral closure and Lipschitz Equisingularity}

For the complex case, in \cite{SGP} the authors obtained conditions so that the canonical vector field defined in a family of simple germs of matrices is Lipschitz, depending of a specific inclusion of ideals, involving the integral closure and the double of an ideal.

A new comprehension of the integral closure in the real case plays a key role in the proof of Theorem \ref{T2.4}. Let us recall this notion.

%\textcolor{red}{We want to show that we get the same conditions holding on the real case, working with notion of real integral closure which we recall next (see \cite{GTW}).}

Let $(\cA_n,m_n)$ be the local ring of real analytic functions germs at the origin in $\bR^n$, and let $\cA_n^p$ be the $\cA_n$-free module of rank $p$. For a germ of a real analytic set $(X,x)$, denote by $\cA_{X,x}$ the local ring of real analytic function germs at $(X,x)$.

\begin{definition}Let $I$ be an ideal of $\cA_{X,x}$. An element $h\in\cA_{X,x}$ is in the \textbf{real integral closure} of $I$, denoted $\overline{I}$, if $h\circ\phi\in \phi^*(I)\cA_1$, for all real analytic path $\phi:(\bR,0)\rightarrow (X,x)$.
\end{definition}

For an algebraic definition of the real integral closure of an ideal one can see \cite{B}.

The key step to obtain the main results of \cite{SGP} for the real case is the fact that the definition of the real integral closure of an ideal is equivalent to the following formulation using analytic inequalities.

\begin{theorem}[\cite{G3}]\label{GR}
	Let $I$ be an ideal of $\cA_{X,x}$ and $h\in \cA_{X,x}$. Then: $h\in\overline{I}$ if and only if for each choice of generators $\{f_i\}$ there exist a positive constant $C$ and a neighborhood $U$ of $x$ such that 
	\begin{center}
		$\parallel h(z)\parallel\leq C \underset{i}{\max}\parallel f_i(z) \parallel $
	\end{center}
\noindent for all $z\in U$.	
\end{theorem}

Let us recall some definitions and fix some notations.

Here we  work with one parameter deformations and unfoldings. The parameter space is denoted by $Y=\bR\equiv \bR\times 0$.

\begin{definition}
Let $h\in\cA_N$. The {\bf double of }$h$ is the element denoted by $h_D\in\cA_{2N}$ defined by the equation $$h_D(z,z'):=h(z)-h(z').$$ 

If $h=(h_1,...,h_r)$ is a map, with $h_i\in\cA_N$, $\forall i$, then we define $I_D(h)$ as the the ideal of $\cA_{2N}$ generated by $\{(h_1)_D,...,(h_r)_D\}$.
\end{definition}

We obtain a relation between the real integral closure of the double and the canonical vector field induced by a one parameter unfolding to be Lipschitz.

Let $\tilde{F}: \bR\times\bR^q\longrightarrow \bR\times\bR^n$ be an analytic map, which is a homeomorphism onto its image, and such that we can write $\tilde{F}(y,x)=(y,\tilde{f}(y,x))$, with $\tilde{f}(y,x)=(\tilde{f}_1(y,x),...,\tilde{f}_n(y,x))$. Let us denote by 

$$\frac{\partial}{\partial y} + \sum\limits_{j=1}^{n}\frac{\partial \widetilde{f_j}}{\partial y}\cdot\frac{\partial}{\partial z_j}$$

\noindent the vector field $v: \tilde{F}(\bR\times\bR^q)\longrightarrow \bR\times\bR^n$ given by $$v(y,z)=(1,\frac{\partial \tilde{f}_1}{\partial y}(\tilde{F}^{-1}(y,z)),...,\frac{\partial \tilde{f}_n}{\partial y}(\tilde{F}^{-1}(y,z))).$$

\begin{theorem}\label{T2.4}
The vector field $\frac{\partial}{\partial y} + \sum\limits_{j=1}^{n}\frac{\partial \widetilde{f}}{\partial y}\cdot\frac{\partial}{\partial z_j}$ is Lipschitz if and only if $$I_D(\frac{\partial \tilde{F}}{\partial y})\subseteq\overline{I_D(\tilde{F})}.$$
\end{theorem}

\begin{proof}

Since we are working in a finite dimensional $\bR$-vector space then all the norms are equivalent. To simplify the argument, we use the notation $\Vert . \Vert$ for the \textit{maximum norm} on $\bR\times\bR^q$ and $\bR\times\bR^n$, i.e, $\Vert (x_1,...,x_{n+1})\Vert = \max_{i=1}^{n+1}\{\Vert x_i\Vert\}$.

Suppose the canonical vector field is Lipschitz. By hypothesis there exists a constant $c>0$ such that $$\parallel v(y,z)-v(y',z') \parallel \leq c\parallel (y,z)-(y',z') \parallel$$ $\forall (y,z),(y',z')\in U$, where $U$ is an open subset of  $\tilde{F}(\bR\times\bR^q)$.

\vspace{0,3cm}

Thus, given $(y,x),(y',x')\in \tilde{F}^{-1}(U)$, and applying the above inequality on these points, we get $$\parallel (\frac{\partial \tilde{f}_j}{\partial y})_D(y,x,y',x') \parallel \leq c\parallel \tilde{F}(y,x)-\tilde{F}(y',x') \parallel $$ for all $j=1,...n$. By the previous theorem, each generator of $I_D(\frac{\partial \tilde{F}}{\partial y})$ belongs to $\overline{I_D(\tilde{F})}$.

\vspace{0,5cm}

Now suppose that $I_D(\frac{\partial \tilde{F}}{\partial y})\subset\overline{I_D(\tilde{F})}$.  Using the hypothesis and  Theorem \ref{GR}, for each $j\in\{1,...n\}$ there exists a constant $c_j>0$ and an open subset $U_j\subset\bR\times\bR^q$ such that $$\parallel (\frac{\partial \tilde{f}_j}{\partial y})_D(y,x,y',x') \parallel \leq c_j\parallel \tilde{F}(y,x)-\tilde{F}(y',x') \parallel $$

$\forall (y,x),(y',x')\in U_j$. Take $U:=\bigcap\limits_{j=1}^{n}U_j$, $c:=\max\{c_j\}_{j=1}^{n}$ and $V:=\tilde{F}(U)$, which is an open subset of $\tilde{F}(\bR\times\bR^q)$, since $\tilde{F}$ is a homeomorphism onto its image. Hence, $$\parallel v(y,z)-v(y',z') \parallel \leq c\parallel (y,z)-(y',z') \parallel$$ $\forall (y,z),(y',z')\in V$.

Therefore, the vector field $\frac{\partial}{\partial y} + \sum\limits_{j=1}^{n}\frac{\partial \tilde{f}_j}{\partial y}\cdot\frac{\partial}{\partial z_j}$ is Lipschitz.
%The same application holds for matrices in the real case.
\end{proof}

%\begin{proof}
%The proof is analogous to the complex case (Proposition 3.3 \cite{SGP}) using the notion of real integral closure and its description via inequalities.
%\end{proof}

\begin{corollary}\label{2.2}
Suppose that $\tilde{F}:\bR\times\bR^q\longrightarrow \bR\times \mbox{Hom}(\bR^m,\bR^n)$ is an analytic map and a homeomorphism onto its image, and suppose we can write $$\tilde{F}(y,x)=(y,F(x)+y\theta(x))$$.

\begin{enumerate}
\item [a)] The vector field $\frac{\partial}{\partial y} + \sum\limits_{j=1}^{n}\frac{\partial \widetilde{f}}{\partial y}\cdot\frac{\partial}{\partial z_j}$ is Lipschitz if, and only if, $$I_D(\theta)\subseteq\overline{I_D(\tilde{F})}.$$

\item [b)] If $\theta$ is constant then the vector field $\frac{\partial}{\partial y} + \sum\limits_{j=1}^{n}\frac{\partial \widetilde{f}}{\partial y}\cdot\frac{\partial}{\partial z_j}$ is Lipschitz.
\end{enumerate} 
\end{corollary}

\section{Applications in some classes of square matrices}

In this section we study if the Lipschitz condition is satisfied on the canonical vector field naturally associated to the 1-unfolding of a $\mathcal{G}$-simple square matrices singularities classified in \cite{Bruce, BruceTari}. Our goal is to obtain a better understanding of its behaviour. In \cite{SGP} we consider versal deformation of determinantal singularities of codimension $2$ and we showed this behaviour depends on the type of the normal form. 

%In \cite{Bruce,BruceTari} the authors determined the pairs $(r,n)$ for which ther exists    $\mathcal{G}$-simple germs$ F: \mathbb{K}^{r} \to G$, where $G$ is $Sym_{n}$ or $Hom(\mathbb{R}^{n},\mathbb{R}^{n})$.

The next result presents a part of the classification of $\mathcal{G}$-simple symmetric matrices obtained by Bruce on Theorem 1.1  of \cite{Bruce}.

\begin{proposition}
The $\mathcal{G}$-simple germs $F: \mathbb{C}^{2} \to Sym_{2}$ of rank $0$ at the origin are given in the following table.
\end{proposition}

\vspace{0.5cm}
\begin{center}
\begin{tabular}{lccc}
	\hline
	   &   Normal Form   &  & Discriminant \\
	\hline
	
	1. & $\left( \begin{matrix}
	y^k  &  x \\
	x     &  y^{\ell}
	\end{matrix}\right)$  & $k\geq1,\ell\geq 2$  & $\cA_{k+\ell +1}$\\  
	
	2. & $\left( \begin{matrix}
	x  &  0 \\
	0     &  y^2+x^k
	\end{matrix}\right)$ & $k\geq2$  & $D_{k+2}$\\  
	
	3. & $\left( \begin{matrix}
	x  &  0 \\
	0     &  xy+y^k
	\end{matrix}\right)$ & $k\geq2$  & $D_{2k}$\\  
	
	4. & $\left( \begin{matrix}
	x  &  y^k \\
	y^k     &  xy
	\end{matrix}\right)$ & $k\geq2$  & $D_{2k+1}$\\  
	
	5. & $\left( \begin{matrix}
	x  &  y^2 \\
	y^2     &  x^2
	\end{matrix}\right)$ &   & $E_6$\\  
	
	6. & $\left( \begin{matrix}
	x  &  0 \\
	0     &  x^2+y^3
	\end{matrix}\right)$ &   & $E_7$\\  
\end{tabular}
\end{center}

In the following result we establish conditions for the Lipschitz triviality of the canonical vector field associated to the normal forms introduced in the above proposition. Differently from the cases exhibited on \cite{SGP}, here we present examples with different nature. Taking the versal deformation of a normal formal we can find  directions that produce Lipschitz trivial deformations, Lipschitz deformations off the origin or non-Lipschitz.

\begin{proposition}
Following the table of normal forms of  $\mathcal{G}$-simple germs $F: \mathbb{C}^{2} \to Sym_{2}$ of rank $0$ at the origin, the canonical vector field associated to the $1$-parameter deformation $\tilde{F}$ induced by $\theta\in\frac{Sym_2}{T\cG_{e}F}$ is Lipschitz in the following conditions: 

\begin{enumerate}

\item[1.]{For the normal form 1 of the table, if the canonical vector field associated to $\tilde{F}$ is Lipschitz then $\theta$ can be written on the form $$\theta=\left(	\begin{matrix}
		a_0+\sum\limits_{i=r}^{k-1}a_iy^i   &  0   \\
		0    & b_0+ \sum\limits_{j=r}^{\ell-2}b_jy^j  \\
		\end{matrix}\right)$$\noindent with $a_i,b_j\in\bC$ and $r=\min\{k,\ell\}$.}
\item[2.]{For the normal form 2 of the table, the canonical vector field associated to $\tilde{F}$ is Lipschitz if and only of $\theta$ can be written on the form $\theta=\left(	\begin{matrix}
		a   &  b   \\
		b    &  \sum\limits_{i=0}^{k-2}d_ix^i  \\
		\end{matrix}\right)$, with $a,b,d_i\in\bC$.}
\item[3.]{For the normal form 3 of the table, the canonical vector field associated to $\tilde{F}$ is Lipschitz if and only of $\theta$ is constant.}

\item[4.]{For the normal form 4 of the table,the canonical vector field associated to $\tilde{F}$ is Lipschitz if and only of $\frac{\partial\tilde{F}}{\partial y}=\frac{\partial F}{\partial y}$, i.e, $\theta$ can be written on the form $\theta=\left(	\begin{matrix}
	a   &  b   \\
	b    &  \sum\limits_{j=0}^{k-1}b_jx^j  \\
	\end{matrix}\right)$, with $a,b,b_j\in\bC$.}

\item[5.]{For the normal form 5 of the table, the canonical vector field associated to the $1$-parameter deformation $\tilde{F}$ induced by $\theta\in\frac{Sym_2}{T\cG_{e}F}$ is Lipschitz if and only if the $1$-jet type of $\tilde{F}$ and $F$ agree.}

\item[6.]{For the normal form 6 of the table, the canonical vector field associated to $\tilde{F}$ is Lipschitz if and only of $\theta$ is constant.}
\end{enumerate}

\end{proposition}

The proof follows from the following lemmas.
\newpage
\begin{lemma}
	Let $F: (\mathbb{C}^2,0)\to Sym_2$ be a $\cG$-simple germ of rank $0$ at the origin whose discriminant of type $\cA_{k+\ell-1}$. Let $\tilde{F}$ be a deformation induced by $\theta\in\frac{Sym_2}{T\cG_{e} F}$. If the canonical vector field associated to $\tilde{F}$ is Lipschitz then $\theta$ can be written on the form \\$\theta=\left(	\begin{matrix}
		a_0+\sum\limits_{i=r}^{k-1}a_iy^i   &  0   \\
		0    & b_0+ \sum\limits_{j=r}^{\ell-2}b_jy^j  \\
		\end{matrix}\right)$, with $a_i,b_j\in\bC$ and $r=\min\{k,\ell\}$.
		
		In particular, in the case $\ell=k$, the canonical vector field associated to $\tilde{F}$ is Lipschitz if and only if $\theta$ is constant.
\end{lemma}

\begin{proof}
	The normal form of $F$ is $\left(
	\begin{matrix}
	y^k   &  x   \\
	x     & y^{\ell}  \\
	\end{matrix}\right)$.
	
%	The extended $\cG$-tangent space of $F$, $T\cG_{e}F$, is generated by
	
%	$$\left\{\left(	\begin{matrix} 
%	0   &  1   \\ 
%	1    &  0  \\ 
%	\end{matrix}\right),\left(	\begin{matrix}
%	ky^{k-1}   &  0   \\
%	0     & \ell y^{\ell-1}  \\
%	\end{matrix}\right),\left(	\begin{matrix}
%	2y^k   &  x   \\
%	x     & 0  \\
%	\end{matrix}\right),\left(	\begin{matrix}
%	0   &  y^k   \\
%	y^k     & 2x  \\
%	\end{matrix}\right),\left(	\begin{matrix}
%	0   &  x   \\
%	x     & 2y^{\ell}  \\
%	\end{matrix}\right),\left(	\begin{matrix}
%	2x   &  y^{\ell}   \\
%	y^{\ell}     & 0  \\
%	\end{matrix}\right)\right\}.$$
	
	Then, the normal space $\frac{Sym_2}{T\cG_{e}F}$ is generated by 
	
	$$\left\{\left(	\begin{matrix}
	1   &  0   \\
	0     & 0  \\
	\end{matrix}\right),\left(	\begin{matrix}
	0   &  0   \\
	0     & 1  \\
	\end{matrix}\right),\left(	\begin{matrix}
	y   &  0   \\
	0     & 0  \\
	\end{matrix}\right),...,\left(	\begin{matrix}
	y^{k-1}   &  0   \\
	0     & 0  \\
	\end{matrix}\right),\left(	\begin{matrix}
	0  &  0   \\
	0     & y  \\
	\end{matrix}\right),...,\left(	\begin{matrix}
	0   &  0   \\
	0     & y^{\ell-2}  \\
	\end{matrix}\right)\right\}.$$
	
	If $\theta \in\frac{Sym_2}{T\cG_{e} F}$ then $\theta$ is a $\bC$-linear combination of the above elements, i.e, there exist $a_i,b_j\in\bC$ such that $$\theta=\left(	\begin{matrix}
	\sum\limits_{i=0}^{k-1}a_iy^i   &  0   \\
	0     & \sum\limits_{j=0}^{\ell-2}b_jy^j  \\
	\end{matrix}\right).$$
	
	Thus, $\tilde{F}=\left(	\begin{matrix}
	y^k+t\sum\limits_{i=0}^{k-1}a_iy^i   &  x   \\
	x     & y^k+t\sum\limits_{j=0}^{k-2}b_jy^j  \\
	\end{matrix}\right)$. 
	
	Notice that $I_D(\tilde{F})$ is generated by $$\{x-x', y^k-y'^k+t\sum\limits_{i=1}^{k-1}a_i(y^i-y'^i),y^{\ell}-y'^{\ell}+t\sum\limits_{j=1}^{\ell-2}b_j(y^j-y'^j)\}$$ and $I_D(\theta)$ is generated by $\left\{\sum\limits_{i=1}^{k-1}a_i(y^i-y'^i),\sum\limits_{j=1}^{\ell-2}b_j(y^j-y'^j)\right\}$.
	
	Consider the curve $\phi(s)=(s^{k+\ell},2s^{k+\ell},2s,s^{k+\ell},s^{k+\ell},s)$. Thus, $\phi^*(I_D(\tilde{F}))=\langle  s^{k+\ell}, (2^k-1)s^k+s^{k+\ell}\sum\limits_{i=1}^{k-1}a_i(2^i-1)s^i, (2^{\ell}-1)s^{\ell}+s^{k+\ell}\sum\limits_{j=1}^{\ell-2}b_j(2^j-1)s^j   \rangle$ which is contained in $\langle s^r \rangle$. Since $I_D(\theta)\subseteq\overline{I_D(\tilde{F})}$ then $$\langle        \sum\limits_{i=1}^{k-1}a_i(2^i-1)s^i,  \sum\limits_{j=1}^{\ell-2}b_j(2^j-1)s^j   \rangle\subseteq\langle s^r\rangle$$ which finishes the proof. 
\end{proof}

\begin{lemma}
		Let $F: (\mathbb{C}^2,0)\to Sym_2$ be a $\cG$-simple germ of rank $0$ at the origin whose discriminant of type $\cD_{k+2}$, $k\geq 2$. Let $\tilde{F}$ be a deformation induced by $\theta\in\frac{Sym_2}{T\cG_{e} F}$. Then the canonical vector field associated to $\tilde{F}$ is Lipschitz if and only of $\theta$ can be written on the form $\theta=\left(	\begin{matrix}
		a   &  b   \\
		b    &  \sum\limits_{i=0}^{k-2}d_ix^i  \\
		\end{matrix}\right)$, with $a,b,d_i\in\bC$.
\end{lemma}

\begin{proof}
	 The normal form of $F$ is $\left(
	\begin{matrix}
	x   &  0   \\
	0     & y^2+x^k  \\
	\end{matrix}\right)$.
	
%	The extended $\cG$-tangent space of $F$, $T\cG_{e}F$, is generated by
	
%	$$\left\{\left(	\begin{matrix}
%	1   &  0   \\
%	0    &  kx^{k-1}  \\
%	\end{matrix}\right),\left(	\begin{matrix}
%	0   &  0   \\
%	0     & y  \\
%	\end{matrix}\right),\left(	\begin{matrix}
%	x   &  0   \\
%	0     & 0  \\
%	\end{matrix}\right),\left(	\begin{matrix}
%	0   &  x   \\
%	x     & 0  \\
%	\end{matrix}\right),\left(	\begin{matrix}
%	0   &  y^2+x^k   \\
%	y^2+x^k     & 0  \\
%	\end{matrix}\right),\left(	\begin{matrix}
%	0   &  0   \\
%	0     & y^2+x^k  \\
%	\end{matrix}\right)\right\}.$$
	
	Then, the normal space $\frac{Sym_2}{T\cG_{e}F}$ is generated by 
	
	$$\left\{\left(	\begin{matrix}
	1   &  0   \\
	0     & 0  \\
	\end{matrix}\right),\left(	\begin{matrix}
	0   &  1   \\
	1     & 0  \\
	\end{matrix}\right),\left(	\begin{matrix}
	0   &  y   \\
	y     & 0  \\
	\end{matrix}\right),\left(	\begin{matrix}
	0   &  0   \\
	0     & 1  \\
	\end{matrix}\right),\left(	\begin{matrix}
	0  &  0   \\
	0     & x  \\
	\end{matrix}\right),...,\left(	\begin{matrix}
	0   &  0   \\
	0     & x^{k-2}  \\
	\end{matrix}\right)\right\}.$$
	
	Thus, we can write $\theta=\left(	\begin{matrix}
	a   &  b+cy   \\
	b+cy    &  \sum\limits_{i=0}^{k-2}d_ix^i  \\
	\end{matrix}\right)$, with $a,b,c,d_i\in\bC$,
	
	$I_D(\theta)=\langle c(y-y'),\sum\limits_{i=1}^{k-2}d_i(x^i-x'^i)  \rangle$ and 	\\$I_D(\tilde{F})=\langle x-x', tc(y-y'),y^2-y'^2+x^k-x'^k+t\sum\limits_{i=1}^{k-2}(x^i-x'^i)  \rangle$.
	
	Consider the curve $\phi(s)=(s,2s^2,2s,s,s^2,s)$. Notice that $\phi^*(I_D(\tilde{F}))\\=\langle s^2,cs^2,3s^2+(2^k-1)s^{2k}+s\sum\limits_{i=1}^{k-2}d_i(2^i-1)s^i \rangle\subseteq \langle s^2 \rangle $.
	
	Suppose the canonical vector field is Lipschitz, i.e, $I_D(\theta)\subseteq\overline{I_D(\tilde{F})}$. Then, $cs=\phi^*(c(y-y'))\in\langle s^2\rangle$ and so $c=0$.
	
	Conversely, if $c=0$ then $I_D(\theta)=\langle \sum\limits_{i=1}^{k-2}d_i(x^i-x'^i)  \rangle\subseteq \langle x-x'\rangle\subseteq I_D(\tilde{F})$.
	\end{proof}

\begin{lemma}
	Let $F: (\mathbb{C}^2,0)\to Sym_2$ be a $\cG$-simple germ of rank $0$ at the origin whose discriminant of type $\cD_{2k}$, $k\geq 2$. Let $\tilde{F}$ be a deformation induced by $\theta\in\frac{Sym_2}{T\cG_{e} F}$. Then the canonical vector field associated to $\tilde{F}$ is Lipschitz if and only of $\theta$ is constant.
\end{lemma}

\begin{proof}
	The normal form of $F$ is $\left(
	\begin{matrix}
	x   &  0   \\
	0     & xy+y^k  \\
	\end{matrix}\right)$.
	
%	The extended $\cG$-tangent space of $F$, $T\cG_{e}F$, is generated by
	
%	$$\left\{\left(	\begin{matrix}
%	1   &  0   \\
%	0    &  y  \\
%	\end{matrix}\right),\left(	\begin{matrix}
%	0   &  0   \\
%	0     & x+ky^{k-1}  \\
%	\end{matrix}\right),\left(	\begin{matrix}
%	x   &  0   \\
%	0     & 0  \\
%	\end{matrix}\right),\left(	\begin{matrix}
%	0   &  x   \\
%	x     & 0  \\
%	\end{matrix}\right),\left(	\begin{matrix}
%	0   &  xy+y^k   \\
%	xy+y^k     & 0  \\
%	\end{matrix}\right),\left(	\begin{matrix}
%	0   &  0   \\
%	0     & xy+y^k  \\
%	\end{matrix}\right)\right\}.$$
	
	Then, the normal space $\frac{Sym_2}{T\cG_{e}F}$ is generated by 
	
	$$\left\{\left(	\begin{matrix}
	1   &  0   \\
	0     & 0  \\
	\end{matrix}\right),\left(	\begin{matrix}
	0   &  0   \\
	0     & 1  \\
	\end{matrix}\right)\left(	\begin{matrix}
	0   &  1   \\
	1     & 0  \\
	\end{matrix}\right),\left(	\begin{matrix}
	y   &  0   \\
	0     & 0  \\
	\end{matrix}\right),...,\left(	\begin{matrix}
	y^{k-2}   &  0   \\
	0     & 0  \\
	\end{matrix}\right),\left(	\begin{matrix}
	0  &  0   \\
	0     & y  \\
	\end{matrix}\right),...,\left(	\begin{matrix}
	0   &  0   \\
	0     & y^{k-1}  \\
	\end{matrix}\right)\right\}.$$
	
	So we can write  $\theta=\left(	\begin{matrix}
	\sum\limits_{i=0}^{k-2}a_iy^i   &  a   \\
	a     & \sum\limits_{j=0}^{k-1}b_jy^j  \\
	\end{matrix}\right)$, for some $a,a_i,b_j\in\bC$, 	$I_D(\theta)=\langle \sum\limits_{i=1}^{k-2}a_i(y^i-y'^i),\sum\limits_{j=1}^{k-1}b_j(y^j-y'^j)  \rangle$ and 	\\$I_D(\tilde{F})=\langle x-x'+t\sum\limits_{i=1}^{k-2}a_i(y^i-y'^i) ,xy-x'y'+y^k-y'^k+t\sum\limits_{j=1}^{k-1}b_j(y^j-y'^j)  \rangle$.
	
	Consider the curve $\phi(s)=(s^k,2s^k,2s,s^k,s^k,s)$. Then $\phi^*(I_D(\tilde{F}))\\=\langle s^k+s^k\sum\limits_{i=1}^{k-2}a_i(2^i-1)s^i,3s^{k+1}+(2^k-1)s^k+s^k\sum\limits_{j=1}^{k-1}b_j(2^j-1)s^j  \rangle\subseteq\langle s^k\rangle$.
	
	If the canonical vector field is Lipschitz then $\sum\limits_{i=1}^{k-2}a_i(2^i-1)s^i$ and $\sum\limits_{j=1}^{k-1}b_j(2^j-1)s^j$ belong to $\langle s^k\rangle$. Hence, $a_i=0$ and $b_j=0$ for all $i$ and $j$. Therefore, $\theta$ is constant.
	
\end{proof}

\begin{lemma}
	Let $F: (\mathbb{C}^2,0)\to Sym_2$ be a $\cG$-simple germ of rank $0$ at the origin whose discriminant of type $\cD_{2k+1}$, $k\geq 2$. Let $\tilde{F}$ be a deformation induced by $\theta\in\frac{Sym_2}{T\cG_{e} F}$. Then the canonical vector field associated to $\tilde{F}$ is Lipschitz if and only of $\frac{\partial\tilde{F}}{\partial y}=\frac{\partial F}{\partial y}$, i.e, $\theta$ can be written on the form $\theta=\left(	\begin{matrix}
	a   &  b   \\
	b    &  \sum\limits_{j=0}^{k-1}b_jx^j  \\
	\end{matrix}\right)$, with $a,b,b_j\in\bC$.
\end{lemma}

\begin{proof}
	The normal form of $F$ is $\left(
	\begin{matrix}
	x   &  y^k   \\
	y^k     & xy  \\
	\end{matrix}\right)$.
	
%	The extended $\cG$-tangent space of $F$, $T\cG_{e}F$, is generated by
	
%	$$\left\{\left(	\begin{matrix}
%	1   &  0   \\
%	0    &  y  \\
%	\end{matrix}\right),\left(	\begin{matrix}
%	0   &  ky^{k-1}   \\
%	ky^{k-1}    & x  \\
%	\end{matrix}\right),\left(	\begin{matrix}
%	2x   &  y^k   \\
%	y^k     & 0  \\
%	\end{matrix}\right),\left(	\begin{matrix}
%	0   &  x   \\
%	x     & 2y^k  \\
%	\end{matrix}\right),\left(	\begin{matrix}
%	2y^k   &  xy   \\
%	xy     & 0  \\
%	\end{matrix}\right),\left(	\begin{matrix}
%	0   &  y^k   \\
%	y^k     & 2xy  \\
%	\end{matrix}\right)\right\}.$$
	
	Then, the normal space $\frac{Sym_2}{T\cG_{e}F}$ is generated by 
	
	$$\left\{\left(	\begin{matrix}
	1   &  0   \\
	0     & 0  \\
	\end{matrix}\right),\left(	\begin{matrix}
	0   &  1   \\
	1     & 0  \\
	\end{matrix}\right),\left(	\begin{matrix}
	0   &  0   \\
	0     & 1  \\
	\end{matrix}\right),\left(	\begin{matrix}
	y   &  0   \\
	0     & 0  \\
	\end{matrix}\right),...,\left(	\begin{matrix}
	y^{k-1}   &  0   \\
	0     & 0  \\
	\end{matrix}\right),\left(	\begin{matrix}
	0  &  0   \\
	0     & x  \\
	\end{matrix}\right),...,\left(	\begin{matrix}
	0   &  0   \\
	0     & x^{k-1}  \\
	\end{matrix}\right)\right\}.$$
	
	Thus, we can write $\theta=\left(	\begin{matrix}
	a+\sum\limits_{i=1}^{k-1}a_iy^i   &  b   \\
	b    &  \sum\limits_{j=0}^{k-1}b_jx^j  \\
	\end{matrix}\right)$, \\with $a,a_i,b,b_j\in\bC$,
	
	$I_D(\theta)=\langle \sum\limits_{i=1}^{k-1}a_i(y^i-y'^i),\sum\limits_{j=1}^{k-1}b_j(x^j-x'^j)  \rangle$ and 	\\$I_D(\tilde{F})=\langle x-x'+t\sum\limits_{i=1}^{k-1}a_i(y^i-y'^i), y^k-y'^k,xy-x'y'+t\sum\limits_{j=1}^{k-1}b_j(x^j-x'^j)  \rangle$.
	
	Consider the curve $\phi(s)=(s^k,2s^k,2s,s^k,s^k,s)$. Then $\phi^*(I_D(\tilde{F}))=\langle s^k+s^k\sum\limits_{i=1}^{k-1}a_i(2^i-1)s^i, (2^k-1)s^k,3s^{k+1}+s^k\sum\limits_{j=1}^{k-1}b_j(2^j-1)s^{kj}  \rangle\subseteq\langle s^k\rangle$.
	
	If $I_D(\theta)\subseteq\overline{I_D(\tilde{F})}$ then $\sum\limits_{i=1}^{k-1}a_i(2^i-1)s^i\in\langle s^k\rangle$, hence $a_i=0,\forall i\in\{1,...,k-1\}$. Conversely, if $a_i=0,\forall i\in\{1,...,k-1\}$ then $I_D(\theta)=\langle \sum\limits_{j=1}^{k-1}b_j(x^j-x'^j)  \rangle\subseteq\langle x-x'\rangle\subseteq I_D(\tilde{F})$. 	
\end{proof}

\begin{lemma}
	Let $F: (\mathbb{C}^2,0)\to Sym_2$ be a $\cG$-simple germ of rank $0$ at the origin with discriminant of type $E_6$. Then the canonical vector field associated to the $1$-parameter deformation $\tilde{F}$ induced by $\theta\in\frac{Sym_2}{T\cG_{e}F}$ is Lipschitz if and only if the $1$-jet type of $\tilde{F}$ and $F$ agree.  
\end{lemma}

\begin{proof}
The normal form of $F$ is $\left(	\begin{matrix}
x   &  y^2   \\
y^2     & x^2  \\
\end{matrix}\right)$. %The extended tangent space $T\cG_{e}F$ is generated by

%$$\left\{\left(	\begin{matrix}
%1   &  0   \\
%0    &  2x  \\
%\end{matrix}\right),\left(	\begin{matrix}
%0   &  y   \\
%y     & 0  \\
%\end{matrix}\right),\left(	\begin{matrix}
%2x   &  y^2   \\
%y^2     & 0  \\
%\end{matrix}\right),\left(	\begin{matrix}
%0   &  x   \\
%x     & 2y^2  \\
%\end{matrix}\right),\left(	\begin{matrix}
%2y^2   &  x^2   \\
%x^2     & 0  \\
%\end{matrix}\right),\left(	\begin{matrix}
%0   &  y^2   \\
%y^2     & 2x^2  \\
%\end{matrix}\right)\right\}.$$ 

Then, the normal space $\frac{Sym_2}{T\cG_{e}F}$ is generated by 

$$\left\{\left(	\begin{matrix}
1   &  0   \\
0     & 0  \\
\end{matrix}\right),\left(	\begin{matrix}
0   &  0   \\
0     & 1  \\
\end{matrix}\right),\left(	\begin{matrix}
0   &  1   \\
1     & 0  \\
\end{matrix}\right),\left(	\begin{matrix}
y   &  0   \\
0     & 0  \\
\end{matrix}\right),\left(	\begin{matrix}
0  &  0   \\
0     & y  \\
\end{matrix}\right),\left(	\begin{matrix}
0   &  0   \\
0     & y^2  \\
\end{matrix}\right)\right\}.$$

If $\theta\in\frac{Sym_2}{T\cG_{e}F}$ induces a non-trivial deformation $\tilde{F}$ then we can write $$\theta(x,y)=\left(	\begin{matrix}
a_1+a_3y+a_4y^2   &  0   \\
0     & a_2+a_5y+a_6y^2  \\
\end{matrix}\right).$$

Thus $\tilde{F}=\left(	\begin{matrix}
x+t(a_1+a_3y+a_4y^2)   &  y^2   \\
y^2     & x^2+t(a_2+a_5y+a_6y^2)  \\
\end{matrix}\right)$.

Notice that 

$I_D(\theta)=\langle a_3(y-y')+a_4(y^2-y'^2),a_5(y-y')+a_6(y^2-y'^2) \rangle$.

Suppose the $1$-jet type of $\tilde{F}$ and $F$ agree. Then $a_3=a_5=0$ and in this case $I_D(\theta)=\langle a_4(y^2-y'^2),a_6(y^2-y'^2) \rangle$. Since $y^2-y'^2\in I_D(\tilde{F})$ then $I_D(\theta)\subseteq I_D(\tilde{F})$ and the canonical vector field is Lipschitz.

Conversely, if the canonical vector field is Lipschitz then $a_3=a_5=0$. In fact, we are assuming that $I_D(\theta)\subseteq \overline{I_D(\tilde{F})}$. 

We have $I_D(\tilde{F})$ is generated by $$\{ y^2-y'^2, x-x'+t(a_3(y-y')+a_4(y^2-y'^2)), x^2-x'^2+t(a_5(y-y')+a_6(y^2-y'^2)) \}.$$

Consider the curve  $\phi(s)=(s,2s^3,2s^2,s,s^3,s^2)$. Then we have that $\phi^*(I_D(\tilde{F}))=\langle 3s^4,s^3+s(a_3s^2+3a_4s^4),3s^6+s(a_5s^2+3a_6s^4)  \rangle\subseteq \langle s^3 \rangle$. Since $\phi^*(I_D(\theta))\subseteq\phi^*(I_D(\tilde{F}))\subseteq\langle s^3\rangle$ then $\phi^*(a_3(y-y')+a_4(y^2-y'^2))\\=a_3s^2+3a_4s^4\in\langle s^3\rangle$ which implies that $a_3s^2\in\langle s^3\rangle$, hence $a_3=0$. Analogously, using the same curve, we prove that $a_5=0$.
\end{proof}

\newpage
	
\begin{lemma}
		
		Let $F: (\mathbb{C}^2,0)\to Sym_2$ be a $\cG$-simple germ of rank $0$ at the origin whose discriminant of type $E_7$. Let $\tilde{F}$ be a deformation induced by $\theta\in\frac{Sym_2}{T\cG_{e} F}$. Then the canonical vector field associated to $\tilde{F}$ is Lipschitz if and only of $\theta$ is constant.
\end{lemma}

	\begin{proof}	
			The normal form of $F$ is $\left(
			\begin{matrix}
			x   &  0   \\
			0     & x^2+y^3  \\
			\end{matrix}\right)$.
		
		%The extended $\cG$-tangent space of $F$, $T\cG_{e}F$, is generated by
		
	%	$$\left\{\left(	\begin{matrix}
	%	1   &  0   \\
	%	0    &  2x  \\
	%	\end{matrix}\right),\left(	\begin{matrix}
	%	0   &  0   \\
	%	0     & 3y^2  \\
	%	\end{matrix}\right),\left(	\begin{matrix}
	%	x   &  0   \\
	%	0     & 0  \\
	%	\end{matrix}\right),\left(	\begin{matrix}
	%	0   &  x   \\
	%	x     & 0  \\
	%	\end{matrix}\right),\left(	\begin{matrix}
	%	0   &  x^2+y^3   \\
	% %	x^2+y^3     & 0  \\
	%	\end{matrix}\right),\left(	\begin{matrix}
	%	0   &  0   \\
	%	0     & x^2+y^3  \\
	%	\end{matrix}\right)\right\}.$$
		
		Then, the normal space $\frac{Sym_2}{T\cG_{e}F}$ is generated by 
		
		$$\left\{\left(	\begin{matrix}
		1   &  0   \\
		0     & 0  \\
		\end{matrix}\right),\left(	\begin{matrix}
		0   &  0   \\
		0     & 1  \\
		\end{matrix}\right)\left(	\begin{matrix}
		0   &  1   \\
		1     & 0  \\
		\end{matrix}\right),\left(	\begin{matrix}
		0   &  0   \\
		0     & y  \\
		\end{matrix}\right),\left(	\begin{matrix}
		y   &  0   \\
		0     & 0  \\
		\end{matrix}\right),\left(	\begin{matrix}
		0  &  y   \\
		y     & 0  \\
		\end{matrix}\right),\left(	\begin{matrix}
		0   &  y^2   \\
		y^2     & 0  \\
		\end{matrix}\right)\right\}.$$
		
		So we can write  $\theta=\left(	\begin{matrix}
		a_1+a_5y   &  a_3+a_6y+a_7y^2   \\
		a_3+a_6y+a_7y^2     & a_2+a_4y  \\
		\end{matrix}\right)$, for some $a_i\in\bC$, 	$I_D(\theta)=\langle a_5(y-y'),a_4(y-y'),a_6(y-y')+a_7(y^2-y'^2)  \rangle$ and 	\\$I_D(\tilde{F})=\langle x-x'+ta_5(y-y') ,t(a_6(y-y')+a_7(y^2-y'^2)),x^2-x'^2+y^3-y'^3+ta_4(y-y')  \rangle$. Consider the curve $\phi(s)=(s^2,2s^3,2s,s^2,s^3,s)$. Thus, $\phi^*(I_D(\tilde{F}))=\langle s^3+a_5s^3,a_6s^3+3a_7s^4,3s^6+7s^3+a_4s^3\rangle\subseteq\langle s^3 \rangle$.
		
		If the canonical vector field is Lipschitz then $a_5s,a_4s,a_6s+3a_7s^2\in\langle s^3\rangle$ which implies that $a_4=a_5=a_6=a_7=0$. Therefore, $\theta$ is constant.	
	\end{proof}
	
%\textcolor{red}{FALTA UM COMENTARIO AQUI}

As in \cite{SGP}, the canonical vector field associated to the $1$-parameter deformation $\tilde{F}$ of the normal forms presented in \cite{Bruce} induced by $\theta\in\frac{Sym_3}{T\cG_{e}F}$ is Lipschitz if and only if the $1$-jet type of $\tilde{F}$ and $F$ agree.  The proof of the next result is analogous to the proof of the main result of \cite{SGP}.

\begin{proposition}
For all  $\mathcal{G}$-simple germs $F: \mathbb{C}^{r} \to Sym_{3}$ of rank $0$ at the origin we have that the canonical vector field associated to the $1$-parameter deformation $\tilde{F}$ induced by $\theta\in\frac{Sym_3}{T\cG_{e}F}$ is Lipschitz. 

\end{proposition}

\begin{proof}
	Suppose that $F$ is of $1$-jet-type of the form in the tables in items (5) and (6) of Theorem 1.1  from \cite{B}. Since $\theta\in\frac{Mat_{(3)}(\cO_r)}{T\mathcal{G}F}$ then the $r$ order $1$ entries  of the matrix $F$ stay unperturbed, thus the differences of the monomial generators of the maximal ideal are in $I_D(\tilde{F})$. In particular the ideal $I_\Delta$ from the diagonal satisfies the inclusion $I_\Delta\subseteq I_D(\tilde{F})$. Let $\theta_i$, $i\in\{1,...,6\}$ be the components of $\theta$. Notice that every $(\theta_i)_D$ vanishes on the diagonal $\Delta$ which implies that all the generators of $I_D(\theta)$ belong to $I_\Delta$. Therefore, $I_D(\theta)\subseteq I_\Delta\subseteq I_D(\tilde{F})$ and Proposition 3.4 of \cite{SGP}  ensures the canonical vector field is Lipschitz. 

\end{proof}

\begin{remark}[\cite{Bruce}, Remark 1.2.] In the cases when $r = 2$ and $n = 2, 3$ the $\mathcal{G}$-codimension of the germs and the Milnor number of the discriminant coincide. %These are special cases of a more general result proved in [6].
\end{remark}

The next result is an application of the results of the previous section for the real case. The proof follows the same steps of Theorem 2.8 of \cite{SGP}.

\begin{theorem}
	Consider the $\mathcal{G}$-simple germs $F: \mathbb{R}^{r} \to \textrm{Hom}(\mathbb{R}^{n},\mathbb{R}^{n})$ of rank $0$ at the origin, classified in Theorem 1.1 of \cite{BruceTari}, and consider the semi-universal unfolding $\tilde{F}:\bR\times\bR^{n}\rightarrow \bR\times \mbox{Hom}(\bR^n,\bR^n)$,  where $\theta\in\frac{Mat_{n}(\mathcal{A}_{r})}{T\mathcal{G}_{e}F}$.
	
%	If $F$ is of 1-jet-type $J_{n,n}$ from Theorem 1.1 of \cite{BruceTari} then the canonical vector field is Lipschitz.
	
		If the ideal of $1$-minors of $F$ defines a reduced point then the canonical vector field is Lipschitz.
\end{theorem}

\begin{proof}
	Since the ideal of $1$-minors of $F$ defines a reduced point and $\theta\in\frac{Mat_{n}(\mathcal{A}_{r})}{T\mathcal{G}_{e}F}$ then the $r$ order $1$ entries of $F$ stay unperturbed, thus the differences of the monomial generators of the maximal ideal are in $I_D(\tilde{F})$. Consequentely, $I_{\Delta}\subseteq I_D(\tilde{F})$. Let $\theta_{ij}$ be the components of $\theta$, $i,j\in\{1,\hdots,n\}$. Clearly all $(\theta_{ij})_D$ vanish on $\Delta$. Hence, $I_D(\theta)\subseteq I_{\Delta}$ and the proof is done by Corollary \ref{2.2}.
\end{proof}

\begin{remark}

In \cite{MP} the author obtain sufficient conditions for topological triviality of $1$-parameter deformations of weighted homogeneous matrix $M$ (see Proposition 6.1 and Proposition 6.2 ). Considering the action defined in the Definition \ref{action}, the triviality condition is related to the tangent space to the $\mathcal{G}$-orbit  of $M$. These condition ensure that the canonical vector field is integrable. 

At this point, one way to continue our study is to show that the homeomorphism obtained by integration of the canonical Lipschitz vector fields gives the bi-Lipschitz equivalence of the members of the respective family of square matrix map-germs according to Definition \ref{action}.

%Integrating this vector field we construct a Lipschitz homemorphism between the members of the family.

\end{remark}

\newpage

\begin{small}
	
{\sc Thiago Filipe da Silva
	
	Departamento de Matem\'atica, Universidade Federal do Esp\'irito Santo \\
	Av. Fernando Ferrari, 514 - Goiabeiras, 29075-910 - Vit\'oria - ES, Brazil, thiago.silva@ufes.br}

\vspace{1cm}
	
{\sc Nivaldo de G\'oes Grulha J\'unior 
	
	Instituto de Ci\^encias Matem\'aticas e de Computa\'c\~ao - USP\\
	Av. Trabalhador S\~ao Carlense, 400 - Centro, 13566-590 - S\~ao Carlos - SP, Brazil, njunior@icmc.usp.br}

\vspace{1cm}
	
{\sc Miriam da Silva Pereira 
	
	Departamento de Matem\'atica, Universidade Federal da Para\'iba, 58.051-900  Jo\~ao Pessoa, Brazil, miriam@mat.ufpb.br}

\end{small}


\begin{thebibliography}{99}

\bibitem{Andrada}{\sc C. Andradas, L. Br\"ocker, J. M. Ruiz }, {\it Constructible Sets in Real Geometry}, V. 33, Ergebnisse der Mathematik und ihrer Grenzgebiete. 3. Folge - A Series of Modern Surveys in Mathematics, Springer Science \& Business Media, 2012.
	
\bibitem{Arnold} {\sc V. I. Arnold}, {\it Matrices depending on parameters}. Russian Math. Surveys 26, no. 2, 29-43 (1971).

\bibitem{Bruce}
{\sc J. W. Bruce}, {\it Families of symmetric matrices}, {\sl Moscow
Math. J.}, {\bf 3}, no 2, 335-360 (2003).

\bibitem{BruceTari}
{\sc J. W. Bruce and F. Tari}, {\it Families of square matrices}, Proc. London Math. Soc. (3) 89, 738-762 (2004).

\bibitem{B} G. W. Brumfiel, \textit{Real valuation rings and ideals}, G\'eom\'etrie alg\'ebrique r\'eelle et formes quadratiques, Rennes 1981, Springer Lecture Notes 959, 55-97, (1982).

\bibitem{Damon}
{\sc  J. Damon}, {\it The unfolding and determinancy theorems for subgoups of ${\mathcal{A}}$ and ${\mathcal K}$}, {\sl Memoirs of the
American Mathematical Society}, Providence RI, (1984).

\bibitem{SGP} {\sc T. da Silva, N. G. Grulha Jr. and M. Pereira,} \textit{The Bi-Lipschitz Equisingularity of Essentially Isolated Determinantal Singularities}, Bull Braz Math Soc, New Series 49, 637-645 (2018).

\bibitem{DP} {\sc J. Damon and B. Pike} {\it Solvable groups, free divisors and nonisolated matrix singularities} II: Vanishing topology. Geom. Topol. 18, no. 2, 911-962 (2014).

\bibitem{FN} {\sc A. Fr\"uhbis-Kr\"uger and A. Neumer}, {\it Simple Cohen-Macaulay
Codimension 2 Singularities}, Communications in Algebra, 38:2, 454-495 (2010).

\bibitem{FK} {\sc A. Fr\"uhbis-Kr\"uger}, {\em Classification of Simple Space Curves
Singularities,} {\sl Communications in Algebra,}  27 (8), 3993-4013, (1999).

\bibitem{FR2} {\sc A. Fernandes and M. A. S. Ruas}, \textit{Bilipschitz determinacy of quasihomogeneous germs}, Glasgow Mathematical Journal, 46 (1), 77-82 (2004).

\bibitem{G1} {\sc T. Gaffney}, {\it Bi-Lipschitz equivalence, integral closure and invariants,} Proceedings of the 10th International Workshop on Real and Complex Singularities. Edited by: M. Manoel, Universidade de S\~ao Paulo, M. C. Romero Fuster, Universitat de Valencia, Spain, C. T. C. Wall, University of Liverpool, London Mathematical Society Lecture Note Series (No.380) November (2010).

\bibitem{G2} {\sc T. Gaffney}, {\it The genericity of the infinitesimal Lipschitz condition for hypersurfaces}, J. Singul. 10, 108-123 (2014). 

\bibitem{SG} {\sc T. Gaffney and T. da Silva}, \textit{Infinitesimal Lipschitz conditions on a family of analytic varieties}, arXiv: 1902.03194 [math.AG].

\bibitem{G3} {\sc T. Gaffney}, {\it Integral Closure of Modules and Whitney equisingularity,} Invent. Math. 107, 301-322 (1992).

\bibitem{LT} {\sc M. Lejeune-Jalabert and B. Teissier}, {\it Cl\^oture int\'egrale des id\'eaux et \newline equisingularit\'e}, S\'eminaire Lejeune-Teissier, Centre de Math\'ematiques \'Ecole Polytechnique, (1974) Publ. Inst. Fourier St. Martin d'Heres, F-38402 (1975).

\bibitem{M1} {\sc T. Mostowski}, {\it A criterion for Lipschitz equisingularity.} Bull. Polish Acad. Sci. Math. 37 (1989), no. 1-6, 109-116 (1990).

\bibitem{Miriam} {\sc M. S. Pereira}, {\em Variedades Determinantais e Singularidades de Matrizes},
{\sl Tese de Doutorado}, {\sl ICMC-USP, http://www.teses.usp.br/teses/disponiveis/55/55135/tde-22062010-133339/en.php} (2010).

\bibitem{PA1} {\sc A. Parusi\'nski}, {\it Lipschitz stratification of real analytic sets.} Singularities (Warsaw, 1985), 323-333, Banach Center Publ., 20, PWN, Warsaw, (1988). 

\bibitem{PA2} {\sc A. Parusi\'nski}, {\it Lipschitz properties of semi-analytic sets.} Ann. Inst. Fourier (Grenoble) 38, no. 4, 189-213 (1988).

\bibitem{MP} M. S. Pereira, {\it Properties of G-Equivalence of Matrices}, arXiv:1711.02156. 

\bibitem{PT} {\sc F. Pham and B. Teissier}, {\it Fractions lipschitziennes d'une alg\`ebre analytique complexe et saturation de Zariski}, Centre Math. l'\' Ecole Polytech., Paris (1969). 

\bibitem{Pham} {\sc F. Pham}, {\it Fractions lipschitziennes et saturation de Zariski des alg\`ebres analytiques complexes}. Expos\'e d'un travail fait avec Bernard Teissier. Fractions lipschitziennes d'une alg\`ebre analytique complexe et saturation de Zariski, Centre Math. l\'Ecole Polytech., Paris, 1969. Actes du Congr\`es International des Math\'ematiciens (Nice, 1970), Tome 2, pp. 649-654. Gauthier-Villars, Paris (1971).

\bibitem{RP} {\sc M. A. S. Ruas and M. S. Pereira}, \emph{Codimension two determinantal varieties with isolated singularities.} Math. Scand. 115, no. 2, 161-172 (2014).

\bibitem{Za} {\sc O. Zariski}, {\it General theory of saturation and of saturated local rings. II. Saturated local rings of dimension 1}. Amer. J. Math. 93, 872-964 (1971).

\end{thebibliography}
\end{document}